\documentclass{amsart}

\usepackage{graphicx} 
\usepackage{amssymb}
\usepackage[hyphens]{url} 
\usepackage{amsmath}
\usepackage{cite}
\usepackage{mathrsfs}
\usepackage{hyperref}
\usepackage{amsthm}
\usepackage{tikz-cd}
\usepackage{caption}
\usepackage{float}
\usepackage[utf8]{inputenc}
\usepackage{graphicx}
\usepackage{stmaryrd}
\usepackage{tikz}
\usepackage{tikz-cd} 
\usepackage[all]{xy}
\usepackage{comment}
\usepackage{bm}
\usepackage{comment}
\usepackage{amsmath,amsfonts,amssymb}
\usepackage{tcolorbox}
\usepackage{cleveref}
\usepackage{hyperref}
\usepackage[a4paper]{geometry}
\usepackage{nameref}
\usepackage{mathabx}
\usepackage{enumitem}
\usepackage{xcolor}
\usepackage{nicefrac}
\usepackage{array}
\hypersetup{
	colorlinks=true,
	linkcolor={blue},
	citecolor={blue},
	urlcolor={blue}}

\theoremstyle{plain}
\newtheorem*{mainthm}{Main Theorem}

\newtheorem{thm}{Theorem}[section]
\newtheorem{prop}[thm]{Proposition}

\newtheorem*{conj*}{Conjecture}
\newtheorem{claim}[thm]{Claim}

\newtheorem{lemma}[thm]{Lemma}

\newtheorem*{cor*}{Corollary}

\theoremstyle{definition}
\newtheorem{notation}[thm]{Notation}
\newtheorem{defn}[thm]{Definition}

\newtheorem{stassumption}[thm]{Standing Assumption}

\theoremstyle{remark}
\newtheorem{rmk}[thm]{Remark}

\newtheorem*{rmk*}{Remark}

\newlength{\plarg}
\usetikzlibrary{cd}

\newcommand{\abs}[1]{\left|#1\right|}
\newcommand{\norm}[1]{\left\lVert#1\right\rVert}
\newcommand{\ab}{\mathrm{ab}}

\newcommand{\lepair}{\le}

\newcommand{\coind}[1]{\left\lfloor\!\left\lfloor #1 \right\rfloor\!\right\rfloor} 

\numberwithin{equation}{section}

\title[Homological torsion growth in graphs of free groups]{Homological torsion growth in non-normal chains of graphs of free groups}
\author{Dario Ascari and Jonathan Fruchter}

\begin{document}

\begin{abstract}
    Let $G$ be a hyperbolic group that splits as a graph of free groups with cyclic edge groups, and which is not isomorphic to a free product of free and surface groups. We show that $G$ admits an exhausting, nested sequence of finite-index non-normal subgroups $G\ge G_1 \ge G_2 \ge \cdots$ with exponential homological torsion growth. More specifically, we prove that simultaneously for every prime $p$,
    \[
    \liminf_{n\rightarrow \infty} \frac{\log \abs{ \mathrm{Tor}_p(G_n^{\mathrm{ab}})}}{[G:G_n]} >0,
    \]
    where $\mathrm{Tor}_p(G_n^{\mathrm{ab}}) = \{g \in G_n^{\mathrm{ab}} \;\vert\; g \text{ has order a power of } p\}$.
\end{abstract}

\maketitle

\tableofcontents

\section{Introduction} \label{intro}

This paper is a continuation of our previous work on virtual torsion in the abelianization of hyperbolic groups that split as graphs of (virtually) free groups with (virtually) cyclic edge groups. There, we proved that, apart from free products of free and surface groups, such groups admit an abundance of homological torsion in finite-index subgroups \cite[Theorem A]{us}. In the present paper we expand on this theme, showing that torsion can grow at an exponential rate along carefully chosen towers of (non-normal) finite-index subgroups. \par \smallskip

In recent years, the study of torsion growth in the homology of finite-sheeted covers has been driven by conjectures of Bergeron–Venkatesh \cite{Bergeron2012} and Lück \cite{Lck2002, Lck2013}, which predict exponential growth of homological torsion along cofinal towers of regular covers of closed hyperbolic $3$-manifolds. Despite much work, no finitely presented group is known to exhibit such growth in a residual chain of normal subgroups. \par \smallskip

In many situations homological torsion is known to grow subexponentially. By contrast, genuine lower bounds are less common, and arise after altering the aforementioned problem in various ways. Dropping the assumption of finite presentability, Kar, Kropholler and Nikolov show that torsion in the abelianization can grow faster than any given function \cite[Theorem~3]{nik}. Foregoing trivial intersection of the groups in the tower, Silver–Williams \cite{SW1,SW2} and Raimbault \cite{Raimbault} prove the non-vanishing of normalized logarithmic torsion for certain knot and link complements. Avramidi, Okun and Schreve show that torsion grows exponentially in the second homology of certain right-angled Artin groups \cite[Corollary~3]{Avramidi}. Finally, giving up the normality assumption, Liu \cite[Theorem 1.3]{Liu2019} proved that every uniform lattice $\Gamma<\mathrm{PSL}(2,\mathbb{C})$ admits an exhausting nested sequence of sublattices $\Gamma \ge \Gamma _1 \ge \Gamma_2 \ge \cdots$ for which $2$-torsion grows exponentially, that is,
\[
    \liminf_{i\rightarrow \infty} \frac{\log \abs{ \mathrm{Tor}_2(H_1(\Gamma_n;\mathbb{Z}))}}{[\Gamma:\Gamma_n]} >0.
\]

Our present paper follows in the same vein. Relying on the methods and results of \cite{us}, we prove:

\begin{mainthm}
    \label{mainthm}
    Let $G$ be a hyperbolic group that splits as a graph of virtually free groups with virtually cyclic edge groups. If $G$ is not virtually a free product of free and surface groups, then there is a nested sequence of finite-index subgroups $G\ge G_1 \ge \allowbreak G_2 \ge \cdots$ such that $\bigcap_i G_i=\{1\}$, and such that for every prime $p$,
    \[
    \liminf_{n\rightarrow \infty} \frac{\log \abs{ \mathrm{Tor}_p(G_n^{\mathrm{ab}})}}{[G:G_n]} >0.
    \]
\end{mainthm}

\begin{rmk*}
    Our construction does not provide a uniform lower bound on the rate of exponential growth for every prime. As $p$ ranges over all primes, we do not know whether the limit inferiors in the \hyperref[mainthm]{Main Theorem} can be bounded from below by a common constant $\epsilon>0$.
\end{rmk*}

One appeal of our construction is that it forces homological torsion to grow simultaneously at every prime. We believe that the same should hold for hyperbolic $3$-manifolds, but covers are more difficult to construct in that setting. By contrast, in the setting of normal covers, which is the one of main interest, Bergeron–Venkatesh predict that in hyperbolic $3$-manifolds, homological torsion should grow subexponentially when restricted to a single prime $p$ \cite[Subsection 9.4]{Bergeron2012}. This is consistent with the aforementioned results of Silver–Williams for cyclic covers of knot complements \cite{SW2}, which show that the $p$-part of homology has trivial growth for every prime $p$. \par \smallskip

Before giving a brief explanation of the tactic of our proof, we point out that the main case of interest in \hyperref[mainthm]{Main Theorem} is when $G$ is one–ended. Residually finite free products in which homological torsion grows exponentially, simultaneously at every prime $p$, are relatively easy to construct from groups admitting sufficient virtual homological torsion. We leave this to the reader.

\subsection*{Proof strategy} The argument proceeds by building inductively a nested sequence of finite-index subgroups
\[
G\ge G_1 \ge G_2 \ge \cdots
\]
in which homological torsion gradually accumulates at every prime $p$ while elements of $G$ are excluded one by one. The inductive scheme has a “fractal-like’’ character: at each stage, the global pattern from the previous step is replicated on a larger scale, using building blocks made out of the groups obtained in the previous stages. The key point is to balance these pieces so that new torsion is introduced without destroying the torsion already present, while simultaneously maintaining a controlled ratio between the amount of homological torsion and the index. \par \smallskip

\begin{figure}[h!]
    \centering
    \makebox[\textwidth][c]{\includegraphics[width=1.2\textwidth]{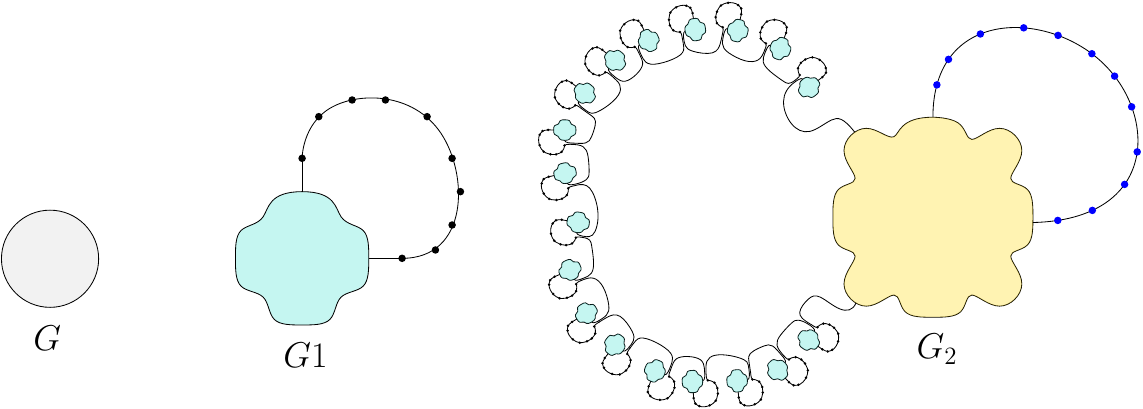}}
    \caption{Constructing $G_0\ge G_1 \ge G_2 \ge \cdots$. At each stage, the dots stand for torsion pieces.}
    \label{fig:intro}
\end{figure}

In \Cref{sec:pieces}, using \cite[Theorem A]{us}, we construct \emph{$p$-torsion pieces}: subgroups of $G$ that contain $p$-homological torsion, and have “loose ends’’ which allow them to be glued together. The group $G_n$ is obtained by modifying a sufficiently deep finite-index subgroup of $G_{n-1}$, attaching two chains: one consisting of copies of the torsion piece for a new prime $p_n$, and another comprised of copies of $G_{n-1}$ itself. The former ensures growth at the new prime $p_n$, while the latter preserves torsion from earlier steps. The essential features of this construction are depicted in \Cref{fig:intro}, which should convey the idea more clearly than words.

\subsection*{Acknowledgments} Ascari was funded by the Basque Government grant IT1483-22. Fruchter received funding from the European Union (ERC, SATURN, 101076148) and the Deutsche Forschungsgemeinschaft (DFG, German Research Foundation) under Germany's Excellence Strategy - EXC-2047/1 - 390685813. 

\section{Preliminaries}
\label{sec:prelim}

In this preliminary section we review the toolkit we will rely on for constructing finite-index subgroups of graphs of free groups with cyclic edge groups. While this exposition is self-contained and includes the ingredients needed for our constructions, we have chosen to keep it streamlined. For a more comprehensive treatment, with intuitive explanations and detailed examples, we refer the reader to our previous paper \cite[Sections 2 and 3]{us}. Throughout the paper we adopt a topological point of view which allows us to phrase constructions in a more geometric language and to draw on familiar tools from covering space theory. \par \smallskip

We adopt Serre’s conventions and regard a graph $\Gamma$ as consisting of a vertex set $\mathrm{V}(\Gamma)$, an edge set $\mathrm{E}(\Gamma)$ equipped with an involution $e\mapsto \overline{e}$ reversing orientation, and maps $\iota,\tau:\mathrm{E}(\Gamma)\to \mathrm{V}(\Gamma)$ recording the initial and terminal vertices of an edge. A \emph{graph of groups} $\mathcal{G}$ over a (connected) graph $\Gamma$ consists of vertex groups $\{G_v\}_{v\in V(\Gamma)}$, edge groups $\{G_e\}_{e\in E(\Gamma)}$ with $G_{\overline{e}}=G_e$, and injective homomorphisms $\varphi_e:G_e\hookrightarrow G_{\tau(e)}$. A \emph{morphism of graphs of groups} $f:\mathcal{H} \longrightarrow \mathcal{G}$ consists of a morphism $\phi:\Delta \longrightarrow \Gamma$ of the underlying graphs, along with morphisms 
\[\{f_v:H_v \rightarrow G_{\phi(v)}\}_{v\in \mathrm{V}(\Delta)}\;\;\text{ and }\;\;\{f_e:H_e\rightarrow G_{\phi(e)}\}_{e\in \mathrm{E}(\Delta)}\]
that intertwine with the edge maps of $\mathcal{H}$ and $\mathcal{G}$.
\par \smallskip

To each graph of groups $\mathcal{G}$, one can attach a fundamental group $G=\pi_1(\mathcal{G})$. A simple way to describe $G$, is by considering the fundamental group of a geometric realization of a \emph{graph of spaces} $\mathcal{X}$ associated to $\mathcal{G}$. We briefly explain how construct such $\mathcal{X}$. Choose a collection of connected CW complexes $\{X_v\}_{v\in \mathrm{V}(\Gamma)}$ and $\{X_e\}_{e\in \mathrm{E}(\Gamma)}$ whose fundamental groups realize the vertex and edge groups of $\mathcal{G}$, along with continuous, combinatorial maps (mapping $k$-cells to $k$-cells) $i_e:X_e\rightarrow X_{\tau(e)}$ such that $(i_e)_\ast=\varphi_e:G_e\rightarrow G_{\tau(e)}$. This data amounts to the graph of spaces $\mathcal{X}$, and its geometric realization $X$ is given by the following quotient
\[
X=\frac{\big(\bigsqcup_{v \in \mathrm{V}(\Gamma)} X_v\big) \; \sqcup \; \big(\bigsqcup_{e\in \mathrm{E}(\Gamma)} X_e\times[0,1]\big)}{\{X_e\times\{0\}\sim X_{\overline{e}}\times\{0\}\}_{e\in \mathrm{E}(\Gamma)}\;\;\text{and}\;\;\{X_e\times \{1\} \sim i_e(X_e)\subseteq X_{\tau(e)}\}_{e\in \mathrm{E}(\Gamma)}}.
\]
In other words, we take the disjoint union of the different vertex and edge spaces of $\mathcal{X}$, and glue each edge space $X_e$ to its image in the terminal vertex space $X_{\tau(e)}$ via the map $i_e$. One easily verifies that $G=\pi_1(X)$ coincides with Serre's definition of $\pi_1(\mathcal{G})$ in \cite[Subsection 5.1]{trees}. We point out that a homomorphism $f:\mathcal{H}\longrightarrow \mathcal{G}$ of graphs of groups induces a homomorphism $f_\ast:H\longrightarrow G$ at the level of fundamental groups. \par \smallskip

Fixing a base point $p\in X$ that lies in one of the vertex spaces $X_{v_p}$ of $\mathcal{X}$, we obtain a base vertex $v_p\in \mathrm{V}(\Gamma)$ of $\mathcal{G}$. To each $g\in G$ we can associate a \emph{reduced path} in the underlying graph $\Gamma$ by considering a geodesic loop $[g]$ in $X$, based at $p$, and projecting it to $\Gamma$ via the map $X\longrightarrow \Gamma$ that collapses each vertex and edge space of $\mathcal{X}$ to a point (cf. Serre's notion of \emph{reduced words} \cite[Subsection 5.2]{trees}). We denote the length of the reduced path of $g$ by $\norm{g}_{\mathcal{G}}$. \par \smallskip

In what follows, we restrict our attention to graphs of groups whose vertex groups are free, and whose edge groups are isomorphic to $\mathbb{Z}$. \par \smallskip

Since our goal is to construct finite-index subgroups of graphs of free groups with cyclic edge groups, we work throughout with finite-sheeted \emph{covers} of such graphs of groups. This lets us use the standard correspondence between subgroups and covers, and phrase our arguments in a topological setting. To set up the necessary tools, it is convenient to adopt a ``local viewpoint'', focusing on each vertex group of $\mathcal{G}$ together with the peripheral data contributed by its incident edges. \par \smallskip

Given a free group $F$ and $w\in F$, we use $[w]_F$ to denote the conjugacy class of $w$ in $F$ (or just $[w]$ when the ambient free group is clear from the context). A \emph{peripheral structure} on a free group $F$ is a (finite) collection of pairwise distinct, non-trivial conjugacy classes $[w_1],\ldots,[w_n]$; we will often abbreviate and write $\underline{w}=(w_1,\ldots,w_n)$ and $[\underline{w}]=\{[w_1],\ldots,[w_n]\}$. We refer to $(F,[\underline{w}])$ as a \emph{pair}. \par \smallskip

Let $G_v$ be a vertex group of $\mathcal{G}$. For each edge $e$ adjacent to $v$, choose a generator $c_e$ of the cyclic group $G_e$. Then 
\[[\underline{w}]=\{[\varphi_e(c_e)]\;\vert\;e\text{ adjacent to }v\}\] 
is a peripheral structure on $G_v$, called the peripheral structure on $G_v$ \emph{induced by} $\mathcal{G}$ at $v$. \par \smallskip

To understand covers of $\mathcal{G}$, we need to understand how peripheral structures behave when passing to subgroups. The key notion is that of \emph{elevations}:

\begin{defn}
    \label{def:alg_elev}
    Let $F$ be a free group, let $H\le F$ and let $1\ne w \in F$. An \emph{elevation} of $[w]_F$ to $H$ is a conjugacy class $[u]_H$ such that
    \begin{enumerate}
        \item there exists $gwg^{-1}\in [w]_F$ such that $u=(gwg^{-1})^d$ for some integer $d\ge 1$, and
        \item $d=\min\{n\;\vert\;g(w^n)g^{-1}\in H\}$.
    \end{enumerate}
    The integer $d$ above is called the \emph{degree} of the elevation, and we denote $d=\deg_{[w]_F}([u]_H)$.
\end{defn}

It is not hard to see that if $H$ is a finitely generated subgroup of $F$, then every conjugacy class in $F$ admits finitely many elevations to $H$ \cite[Lemma 3.8]{us}. The following notion describes how pairs induced at the vertices of a graph of groups behave when passing to covers.

\begin{defn}
    Let $(F,[\underline{w}]_F)$ be a pair and let $H\le F$ be a finitely generated subgroup. Let $[u]_H$ be a peripheral structure on $H$. We say that
    \begin{enumerate}
        \item $(H,[\underline{u}]_H)$ is a \emph{subpair} of $(F,[\underline{w}]_F)$ if every conjugacy class in $[\underline{u}]_H$ is an elevation of a conjugacy class in $[\underline{w}]_F$. In this case, we denote $(H,[\underline{u}]_H)\lepair(F,[\underline{w}]_F)$.
        \item If in addition $H$ is a finite-index subgroup of $F$, and $[\underline{u}]_H$ is the full set of elevations of  $[\underline{w}]_F$ to $H$, then we call the subpair $(H,[\underline{u}]_H)\lepair(F,[\underline{w}]_F)$ a \emph{finite-index pull-back pair}, and refer to $[\underline{u}]_H$ as the \emph{pull-back structure} of $[\underline{w}]_F$ on $H$.
    \end{enumerate}
\end{defn}

With these definitions in hand, we are ready to define finite-sheeted covers of graphs of free groups with cyclic edge groups. 

\begin{defn} \label{def:cover}
    Let $\mathcal{G}$ and $\mathcal{H}$ be graphs of free groups with cyclic edge groups with finite underlying graphs. We say that $\mathcal{H}$ is a \emph{finite-sheeted cover} of $\mathcal{G}$ if there exists a morphism 
    \[f=(\varphi,f_v,f_e):\mathcal{H}\rightarrow \mathcal{G}\]
    such that
    \begin{enumerate}
        \item all of the morphisms $f_v$ and $f_e$ mapping the vertex and edge groups of $\mathcal{H}$ to those of $\mathcal{G}$ are injective, and
        \item for every vertex $v$ of the underlying graph of $\mathcal{H}$, let $(H_v,[\underline{u}])$ be the pair induced by $\mathcal{H}$ at $v$. Then $(f_v(H_v),[f_v(\underline{u})])$ is a finite-index pull-back of the pair $(G_{\phi(v)},[\underline{w}])$ induced by $\mathcal{G}$ at $\phi(v)$.
    \end{enumerate}
\end{defn}

One can readily verify that if $\mathcal{H}$ is a finite-sheeted cover of $\mathcal{G}$ then $H$ is a finite-index subgroup of $G$. The \emph{degree} of the covering is
\[
\deg(\mathcal{H}\twoheadrightarrow \mathcal{G})=[G:H].
\]
On the other hand, if $H$ is a finite-index subgroup of $G$ then $H$ inherits a graph of groups decomposition $\mathcal{H}$ from $\mathcal{G}$, accompanied with a finite-sheeted covering morphism $f:\mathcal{H}\twoheadrightarrow \mathcal{G}$. \par \smallskip

We will often construct finite-sheeted covers of $\mathcal{G}$ “from the ground up.” The idea is to begin with a collection of finite-index pull-back subpairs of the vertex pairs induced by $\mathcal{G}$, and then glue them together along their peripheral structures to obtain a graph of groups as in \Cref{def:cover}. Such a gluing is possible if and only if for every edge $e\in E(\Gamma)$ with edge group $G_e=\langle c_e\rangle$, there exists a degree-preserving bijection between the elevations of $\varphi_e(c_e)$ and of $\varphi_{\overline e}(c_e)$ to the chosen subpairs. \par \smallskip

We will also be interested in infinite-index subgroups of $G$. To this end, it is useful to recall the notion of precovers, which may be regarded as compact cores of certain covers of $\mathcal{G}$:

\begin{defn} \label{def:precover}
    Let $\mathcal{G}$ and $\mathcal{H}$ be graphs of free groups with cyclic edge groups. We say that $\mathcal{H}$ is a \emph{precover} of $\mathcal{G}$ if there exists a morphism 
    \[f=(\varphi,f_v,f_e):\mathcal{H}\rightarrow \mathcal{G}\]
    such that
    \begin{enumerate}
        \item all of the morphisms $f_v$ and $f_e$ mapping the vertex and edge groups of $\mathcal{H}$ to those of $\mathcal{G}$ are injective, and
        \item for every vertex $v$ of the underlying graph of $\mathcal{H}$, let $(H_v,[\underline{u}])$ be the pair induced by $\mathcal{H}$ at $v$. Then $(f_v(H_v),[f_v(\underline{u})])$ is a subpair of the pair $(G_{\phi(v)},[\underline{w}])$ induced by $\mathcal{G}$ at $\phi(v)$.
    \end{enumerate}
\end{defn}

As mentioned earlier, precovers of $\mathcal{G}$ correspond to subgroups of $G$: if $f:\mathcal{H}\to\mathcal{G}$ is a connected precover, then the induced homomorphism $f_\ast:H\to G$ is injective \cite[Lemma 16]{wil:one-ended}. \par \smallskip

In the hyperbolic case, certain precovers can be extended to genuine finite-sheeted covers. Suppose $G$ is hyperbolic and $f:\mathcal{H}\to\mathcal{G}$ is a connected precover in which every vertex group of $\mathcal{H}$ has finite index in the corresponding vertex group of $\mathcal{G}$. Wise showed that hyperbolic graphs of free groups with cyclic edge groups are subgroup separable \cite[Theorem 5.1]{wise:graph-sep}. It follows from Scott's criterion \cite{Scott1978} (see \cite[Lemma 1.3]{wil:sep} for a detailed proof) that $\mathcal{H}$ embeds as a subgraph of some finite-sheeted cover $\mathcal{G}'$ of $\mathcal{G}$. In this case we say that $\mathcal{H}$ can be \emph{completed} to a finite-sheeted cover of $\mathcal{G}$, and call $\mathcal{G}'$ a \emph{completion} of $\mathcal{H}$. \par \smallskip

As with finite-sheeted covers, precovers can also be assembled by gluing together subpairs of the vertex pairs induced by $\mathcal{G}$. The key difference is that we no longer require the full set of elevations to satisfy the aforementioned compatibility conditions. Instead, we may choose to glue only certain elevations whose degrees agree, while leaving others unmatched. Such leftover elevations are called \emph{hanging elevations}. \par \smallskip

The last result that we record deals with controlling the degrees of elevations when passing to finite-index pull-back subpairs. The upshot of obtaining such control, is that it will allow us to splice precovers of $\mathcal{G}$ into finite-sheeted covers of $\mathcal{G}$ even when the degrees of their hanging elevations do not match. A pair $(F,[\underline{w}])$ is malnormal if for every choice of $[w_i]$ and $[w_j]$ in $[\underline{w}]$, and every choice of representatives $w'_i\in [w_i]$ and $w'_j\in [w_j]$, the condition
\[ \langle w'_i\rangle \cap g \langle w'_j \rangle g^{-1}\ne \{1\},\]
forces $w'_i=w'_j$ and $g\in \langle w'_i\rangle$. Equivalently, for some (hence any) choice of representatives $w'_i\in [w_i]$, the tuple $(w'_1,\ldots,w'_n)$ is \emph{independent} (no two elements have conjugate non-trivial powers), and each $w'_i$ generates a maximal cyclic subgroup of $F$. \par \smallskip

The following notion is due to Wise:

\begin{defn}[{\cite[Definition 3.2]{wise:graph-sep}}]
    A group $G$ is called \emph{omnipotent} if, for every tuple of independent non-trivial elements $(g_1,\dots,g_r)$ in $G$, there is an \emph{omnipotence constant} 
    \[K=K(g_1,\dots,g_r)\] 
    such that for every choice of positive integers $(d_1,\dots,d_r)$, there is a finite quotient $q:G\rightarrow Q$ in which the order of $q(g_i)$ is exactly $K\cdot d_i$ for $i=1,\dots,r$.
\end{defn}

Wise proved that free groups are omnipotent \cite[Theorem 3.5]{wise:graph-sep}; before closing this preliminary section we record that a closer look at his proof yields the following refinement, which is the version of omnipotence we will use in our constructions.

\begin{thm} \label{thm:omni}
    Let $(F,[\underline{w}]=\{[w_1],\ldots,[w_n]\})$ be a malnormal pair. Then there is an omnipotence constant $K$, depending only on $[\underline{w}]$, such that for every $n$-tuple of positive integers $\underline{d}=(d_1,\ldots,d_n)$ there is a finite-index pull-back subpair $(G_{\underline{d}},[\underline{u}])\lepair (F,[\underline{w}])$ satisfying
    \begin{enumerate}
        \item $G_{\underline{d}}$ is a finite-index normal subgroup of $F$, and
        \item for every $1\le i \le n$, and for every elevation $[u_j]$ of $[w_i]$ to $G_{\underline{d}}$, 
        \[
        \deg_{[w_i]}[u_j] = K\cdot d_i.
        \]
    \end{enumerate}
    Moreover, if $\underline{d}'=(d'_1,\ldots,d'_n)$ satisfies $d_i \vert d'_i$ for every $1\le i \le n$, then $(G_{\underline{d}'},[\underline{u}'])$ is a finite-index pull-back subpair of $(G_{\underline{d}},[\underline{u}])$.
\end{thm}

\begin{rmk}
    The last part of \Cref{thm:omni} above follows from the fact that if $d_i \mid d'_i$, the corresponding groups $L_r$ in \cite[Lemma 3.11]{wise:graph-sep} contain one another.
\end{rmk}

\medskip

\section{Torsion pieces}
\label{sec:pieces}

In this section we introduce the main building block of our sequence of covers of a hyperbolic graph of free groups with cyclic edge groups $\mathcal{G}$: \emph{$p$-torsion pieces}. Each of these pieces is precovers of $\mathcal{G}$ that contain $p$-torsion in the abelianization and admits hanging elevations. By chaining torsion pieces together, one obtains precovers with abundant $p$-torsion in homology, while still controlling the index of finite-sheeted covers in which they embed. \par \smallskip

Before turning to the constructions, we would like to put $\mathcal{G}$ into \emph{normal form}, a setting that is particularly convenient for gluing arguments; below is a streamlined version of the proposition, including only the specific properties that will be used in our proof:

\begin{prop}
    [\emph{cf. }{\cite[Proposition 3.35]{us}}]
    \label{prop:normal_form}
    Let $G$ be a one-ended hyperbolic group that splits as a graph of free groups with cyclic edge groups $\mathcal{G}$. Then there is a finite-index subgroup $H\le G$ satisfying the following properties:
    \begin{enumerate}
        \item $H$ decomposes as a graph of free groups with cyclic edge groups $\mathcal{H}$.
        \item The vertex groups of $\mathcal{H}$ are either \emph{cyclic} (with vertex group isomorphic to $\mathbb{Z}$) or \emph{non-cyclic}.
        \item \label{item:3-normal} The underlying graph $\Delta$ of $\mathcal{H}$ is bipartite, with edges connecting cyclic vertices to non-cyclic vertices. Moreover, for every edge $e\in \mathrm{E}(\Delta)$, if $\tau(e)$ is cyclic then $\varphi_e:G_e\rightarrow G_{\tau(e)}$ is an isomorphism.
        \item If $v\in \mathrm{V}(\Delta)$ is non-cyclic then the induced pair $(H_v,[\underline{u}])$ is malnormal.
    \end{enumerate}
\end{prop}

\begin{rmk}
    If $\mathcal{G}$ is in normal form, then so are all of its finite-sheeted covers.
\end{rmk}

Note that \Cref{prop:normal_form} assumes $G$ to be one-ended; we begin with a simple reduction that avoids this additional assumption. \Cref{prop:normal_form} implies that $G$ admits a finite-index subgorup $G'$ in which every one-ended factor of its Grushko decomposition $G'=G'_1\ast \cdots \ast G'_n \cdot F$ is in normal form. Picking a non-cyclic vertex group $G'_{v_i}$ of each one-ended factor $G'_i$, and regarding the free product $G'_{v_1}\ast \cdots \ast G'_{v_n} \ast F$ as a single vertex group of $\mathcal{G}'$, we may assume that $\mathcal{G}'$ satisfies the four properties of \Cref{prop:normal_form}. We will therefore work under the following standing assumption:

\begin{stassumption}
     $\mathcal{G}$ is a hyperbolic group that splits as a graph of free groups with cyclic edge groups, which is not isomorphic to a free product of free and surface groups, and which satisfies the properties in \Cref{prop:normal_form}. We keep the notation from Subsection \ref{sec:prelim}; in particular, $\Gamma$ denotes the underlying graph of $\mathcal{G}$.
\end{stassumption}

We will also need the following elementary claims, which will be used in the construction of $p$-torsion pieces as well as in the proof of the \hyperref[mainthm]{Main Theorem}. Recall that a \emph{cut set} of a graph $\Gamma$ is a set $S$ of vertices such that $\Gamma - S$ is disconnected.

\begin{claim}
    \label{claim:connected}
    Let $C$ be a cyclic vertex of $\mathcal{G}$. Then there exists a finite-sheeted cover $\mathcal{G}'$ of $\mathcal{G}$, and lifts $C_1,\ldots, C_n$ of $C$ to $\mathcal{G}'$, such that
    \begin{enumerate}
        \item $C_1,\ldots, C_n$ are all adjacent to the same non-cyclic vertex of $\mathcal{G}'$, and
        \item $\{C_1,\ldots,C_n\}$ is not a cut set of the underlying graph $\Gamma'$ of $\mathcal{G}'$.
    \end{enumerate} 
\end{claim}

\begin{proof}
    Let $e_1,\ldots,e_\ell\in \mathrm{E}(\Gamma)$ be the edges adjacent to $C$ in $\Gamma$ (so that $\tau(e_i)=C$ for $1\le i \le \ell$), and let $v_i=\iota(e_i)$ (we allow $v_i=v_j$ for $i\ne j$). Each $v_i$ is a non-cyclic vertex of $\mathcal{G}$, and the attaching map identifies $G_{e_i}$ with a cyclic subgroup $\varphi_{\overline{e}_i}(G_{e_i})\le G_{v_i}$. Let $t_i$ be a generator of $\varphi_{\overline{e}_i}(G_{e_i})$. Since each $G_{v_i}$ is non-cyclic, choosing an appropriately large finite quotient of $(G_{v_i})^\ab$ in which $t_i$ dies gives a finite-index normal subgroup $t_i \in H_i\trianglelefteq G_{v_i}$ with $[G_{v_i}:H_i]>n$. In particular, $[t_i]$ admits $[G_{v_i}:H_i]>n$ distinct degree-$1$ elevations to $H_i$. \par\smallskip
    
    To finish, take $n+1$ copies $C_1, \ldots, C_{n+1}$ of $C$, and glue $n+1$ of the elevations of $[t_i]$ to $H_i$ to $C_1,\ldots, C_{n+1}$. Since $\mathcal{G}$ is subgroup separable, we may complete this precover of $\mathcal{G}$ to a finite-sheeted cover $\mathcal{G}'$ \cite[Theorem 5.1]{wise:graph-sep}. By construction, removing $C_1,\ldots, C_n$ from $\mathcal{G}'$ does not disconnect $\mathcal{G}'$.
\end{proof}

\begin{claim} \label{claim:lift-cutset}
Let $\mathcal{G}'$ and $C_1,\ldots,C_n$ be as in \Cref{claim:connected}. Then for every finite-sheeted cover $\mathcal{G}^\dagger$ of $\mathcal{G}'$, there exist lifts $C^\dagger_1,\ldots,C^\dagger_n$ of $C_1,\ldots,C_n$ to $\mathcal{G}^\dagger$, such that $\{C^\dagger_1,\ldots,C^\dagger_n\}$ is not a cut set of the underlying graph $\Gamma^\dagger$ of $\mathcal{G}^\dagger$.
\end{claim}

\begin{proof}    
    Let $\mathcal{G}^\dagger$ be a finite-sheeted cover of $\mathcal{G}'$, and denote its underlying graph by $\Gamma^\dagger$. By the construction in the proof of \Cref{claim:connected}, there is a non-cyclic vertex $v\in \mathrm{V}(\Gamma')$ adjacent to $C_1,\ldots,C_n$. For each $1\le i \le n$, let $e_i\in \mathrm{E}(\Gamma')$ be an edge joining $v$ and $C_i$. Choose a vertex $v^\dagger\in \mathrm{V}(\Gamma^\dagger)$ that lies above $v$. For each $1\le i \le n$, there exists an edge $e^{\dagger}_i$ that lies above $e_i$ with initial vertex $v^\dagger$; denote the terminal vertex of $e^\dagger_{i}$ by $C^\dagger_i$. We claim that $\{C^\dagger_1,\ldots,C^\dagger_n\}$ is not a cut set of $\Gamma^\dagger$. \par \smallskip

    We first observe that since $\mathcal{G}'$ is in normal form, the third property of \Cref{prop:normal_form} ensures that the valence of $C^\dagger _i$ in $\Gamma^\dagger$ coincides with the valence of $C_i$ in $\Gamma'$. In particular, $e^\dagger_i$ is the only edge adjacent to $C^\dagger_i$ that lies above $e_i$. \par \smallskip
    
    To show that $\{C^\dagger_1,\ldots,C^\dagger_n\}$ is not a cut set of $\Gamma^\dagger$, it suffices to check that for every vertex $w^\dagger\in \mathrm{V}(\Gamma^\dagger)$ which is adjacent to some $C^\dagger_i$, there is a path between $v^\dagger$ and $w^\dagger$ in $\Gamma^\dagger-\{C^\dagger_1,\ldots,C^\dagger_n\}$.  \par \smallskip

    Let $w\in \mathrm{V}(\Gamma')$ be the vertex lying below $w^\dagger$; note that $(w,C_i)\in \mathrm{E}(\Gamma)$. Since $\{C_1,\ldots,C_n\}$ is not a cut set of $\Gamma$, there exists an embedded loop $\gamma'_w:S^1\rightarrow \Gamma$ that travels through the following vertices (and edges) of $\Gamma'$:
    \begin{equation}
    v \xrightarrow{e_i} C_i \longrightarrow w \longrightarrow u_1 \longrightarrow \cdots \longrightarrow u_k \longrightarrow v
    \end{equation} \label{eq:path}
    where $u_j \notin \{v,w,C_1,\ldots,C_n\}$ for every $1\le j \le k$. Let $g\in G'$ be an element whose associated reduced path in $\Gamma'$ coincides with the path above.
    \par \smallskip

    Regard $v^\dagger$ as the base vertex of $\mathcal{G}^\dagger$; since $G^\dagger$ is a finite-index subgroup of $G'$, $g^d \in G^\dagger$ for some $d\ge 1$. Consider the reduced path $P$ in $\Gamma^\dagger$ associated to $g^d$; up to changing our choice of $g$, we may assume that the path $P$ starts with $v^\dagger \xrightarrow{e^\dagger_i}C^\dagger_i \longrightarrow w^\dagger \longrightarrow \cdots$. \par \smallskip

    If $P$ passes through $C^\dagger_i$ exactly once, then the suffix from its visit to $w^\dagger$ back to $v^\dagger$ provides the desired path. However, $P$ may meet $C^\dagger_i$ multiple times (up to $d$ times). In this case, the valence observation above implies that $P$ must return to $C^\dagger_i$ through the edge $e^\dagger_i$, and hence through $v^\dagger$. In particular, the subpath of $P$ beginning with its first visit to $w^\dagger$ and ending with its next visit to $v^\dagger$ yields the required path that avoids $C^\dagger_1,\ldots, C^\dagger_n$.
\end{proof}

Before constructing $p$-torsion pieces, we fix the following notation: for a finitely generated abelian group $A$, let $b_1(A)$ denote the first Betti number of $A$, i.e., the rank of its free part. For a prime $p$, let $r_p(A)$ be the \emph{$p$-rank} of $A$, that is, the number of $p$-group factors in a primary decomposition of $A$. We have that
\[
r_p(A)=\dim_{\mathbb{F}_p} H_1(A;\mathbb{F}_p)-b_1(A)=\dim_{\mathbb{F}_p}A\otimes \mathbb{F}_p-b_1(A)=\dim_{\mathbb{F}_p}(A/pA)-b_1(A).
\]

The following straightforward lemma records how the $p$-rank behaves under quotients:

\begin{lemma}
    \label{lem:relation}
    Let $A$ be a finitely generated abelian group and let $x\in A$. Then 
    \[r_p(A)\ge r_p(A/\langle x \rangle )-1.\]
\end{lemma}

\begin{proof}
    Consider the short exact sequence
    \[
    0\longrightarrow \langle x \rangle \xrightarrow{\;i\;} A \longrightarrow A/\langle x \rangle\longrightarrow 0.
    \]
    Tensoring with $\mathbb{F}_p$ is right-exact, yielding
    \[
    \langle x \rangle \otimes \mathbb{F}_p \xrightarrow{i\otimes \mathbb{F}_p} A \otimes \mathbb{F}_p \longrightarrow (A/\langle x \rangle ) \otimes \mathbb{F}_p\longrightarrow 0,
    \]
    and therefore
    \begin{align*}
    \dim_{\mathbb{F}_p}(A/pA)&=\dim_{\mathbb{F}_p}((A/\langle x\rangle)/p(A/\langle x \rangle))+\dim_{\mathbb{F}_p}(\mathrm{im}(i\otimes \mathbb{F}_p)) \\ &
    \ge \dim_{\mathbb{F}_p}((A/\langle x\rangle)/p(A/\langle x \rangle)). 
    \end{align*}
    On the other hand, $b_1(A/\langle x \rangle )\ge b_1(A)-1$. Combining these gives
    \begin{align*}
    r_p(A)&=\dim_{\mathbb{F}_p}(A/pA)-b_1(A) \\ &
    \ge \dim_{\mathbb{F}_p}((A/\langle x\rangle)/p(A/\langle x \rangle)) - b_1(A/\langle x \rangle ) - 1 = r_p(A/\langle x \rangle)-1.
    \end{align*}
\end{proof}

With this in hand, we are ready to construct our first building block: $p$-torsion pieces.

\begin{lemma} [$p$-Torsion pieces]
\label{lem:beads}
    Let $C$ be a cyclic vertex of $\mathcal{G}$. Then for every prime $p$ there exists a precover $\mathcal{G}'_p$ of $\mathcal{G}$ satisfying the following properties: 
    \begin{enumerate}
        \item There are two cyclic vertices $C_1$ and $C_2$ of $\mathcal{G}'_p$ that lie above $C$, such that every edge map $\varphi_e:G_e\longrightarrow C$ admits a single hanging elevation to $\mathcal{G}'_p$, with either $C_1$ or $C_2$ as its target vertex. Both $C_1$ and $C_2$ appear as targets of such hanging elevations, and these are the only hanging elevations of $\mathcal{G}'_p$.
        \item $\pi_1(\mathcal{G}'_p)^{\ab}=P\oplus K$, where $P$ is a (non-trivial) finite $p$-group and $[C_1],[C_2]\in K$.
        \end{enumerate}
\end{lemma}

\begin{proof}
    By \Cref{claim:connected}, up to passing to a finite-sheeted cover of $\mathcal{G}$, we may assume that removing $C$ does not disconnect $\mathcal{G}$. By \cite[Theorem A]{us}, there exists a finite-sheeted cover $\mathcal{G}_p$ of $\mathcal{G}$ that has a $(\mathbb{Z}/p\mathbb{Z})^4$ direct factor in its abelianization. By \Cref{claim:lift-cutset}, there is a cyclic vertex $C'$ of $\mathcal{G}_p$ that lies above $C$, and such that removing $C'$ does not disconnect $\Gamma_p$. \par \smallskip
    
    The precover $\mathcal{G}'_p$ is obtained from $\mathcal{G}_p$ in the following manner: remove $C'$ from $\mathcal{G}_p$, and divide the hanging elevations of $\mathcal{G}_p-C'$ into two non-empty sets $E_1$ and $E_2$. Take two copies $C_1$ and $C_2$ of $C'$, and attach the hanging elevations in $E_i$ to $C_i$ for $i=1,2$. Observe that $\pi_1(\mathcal{G}_p)$ coincides with the HNN extension $\mathcal{G}'_p \ast _{C_1=C_2}$, which at the level of homology translates to
    \[
    \pi_1(\mathcal{G}_p)^{\ab}=\frac{{\pi_1(\mathcal{G}'_p})^{\ab}}{[C_1-C_2]} \oplus \mathbb{Z}.
    \]
    In particular, by \Cref{lem:relation},
    \[r_p(\pi_1({\mathcal{G}'_p})^{ab})\ge r_p(\pi_1(\mathcal{G}_p)^\ab)-1=4-1=3.\]
    \par \smallskip
    
    Let $E_p=\bigoplus_{i=1}^m \mathbb{Z}/p^{n_i} \mathbb{Z}$ be the $p$-part of the primary decomposition of $\pi_1({\mathcal{G}'_p})^{\ab}$, and denote the projection $\pi_1({\mathcal{G}'_p})^{\ab} \twoheadrightarrow E_p$ by $\pi_p$. By Smith normal form, up to applying an automorphism of $E_p$, we may assume that the $2$-generated subgroup $\pi_p(\langle [C_1],[C_2]\rangle )$ of $E_p$ lies in $\mathbb{Z}/p^{n_1}\mathbb{Z}\oplus \mathbb{Z}/p^{n_2}\mathbb{Z}$. Set $P=\bigoplus_{i=3}^m \mathbb{Z}/p^{n_i}\mathbb{Z}$, and note that $P$ is non-trivial since $r_p(\pi_1({\mathcal{G}'_p})^{ab})\ge3$. We have that 
    \[\pi_1(\mathcal{G}'_p)^{\ab}=P\oplus K\]
    for some direct factor $K$ containing $[C_1]$ and $[C_2]$.
\end{proof}

\section{Proof of the \hyperref[mainthm]{Main Theorem}}

In this section we prove our main result, restated here for convenience:

\begin{thm}
    [{\hyperref[mainthm]{Main Theorem}}]
    Let $G$ be a hyperbolic group that splits as a graph of virtually free groups with virtually cyclic edge groups. If $G$ is not virtually a free product of free and surface groups, then there is a sequence of nested finite index subgroups $G\ge G_1 \ge G_2 \ge \cdots$ such that $\bigcap_i G_i=\{1\}$ and such that for every prime $p$,
    \[
    \liminf_{i\rightarrow \infty} \frac{\log \abs{ \mathrm{Tor}_p(G_i^{\mathrm{ab}})}}{[G:G_i]} >0.
    \]
\end{thm}

Before turning to the proof, let us describe the general scheme of our inductive construction. Each stage combines two main ingredients: omnipotence (see \Cref{thm:omni}), which allows us to glue precovers with hanging elevations of mismatched degrees into genuine finite-sheeted covers, and the $p$-torsion pieces from \Cref{sec:pieces}, which contribute new $\mathbb{Z}/p\mathbb{Z}$ summands to homology. The group $G_n$ is assembled by attaching two chains to a large precover of $\mathcal{G}_{n-1}$: 
\begin{enumerate}
    \item A chain of copies of a $p_n$-torsion piece constructed out of $\mathcal{G}_{n-1}$ using \Cref{lem:beads}.
    \item A chain of copies of a precover $\mathcal{G}'_{n-1}$ obtained by disconnecting an edge of $\mathcal{G}_{n-1}$.
\end{enumerate}
The first chain introduces a $p_n$-part to the abelianization, whereas the second preserves the torsion constructed in the earlier stages. \par \smallskip

The schematic illustration given in \Cref{fig:intro} of the introduction should be kept in mind, as it captures the essential idea of the construction. For the reader’s convenience, \Cref{fig:construction} below depicts a further iteration of the process, with the various pieces labeled according to the notation used in the proof.

\begin{figure}[h!]
    \centering
    \makebox[\textwidth][c]{\includegraphics[width=1.1\textwidth]{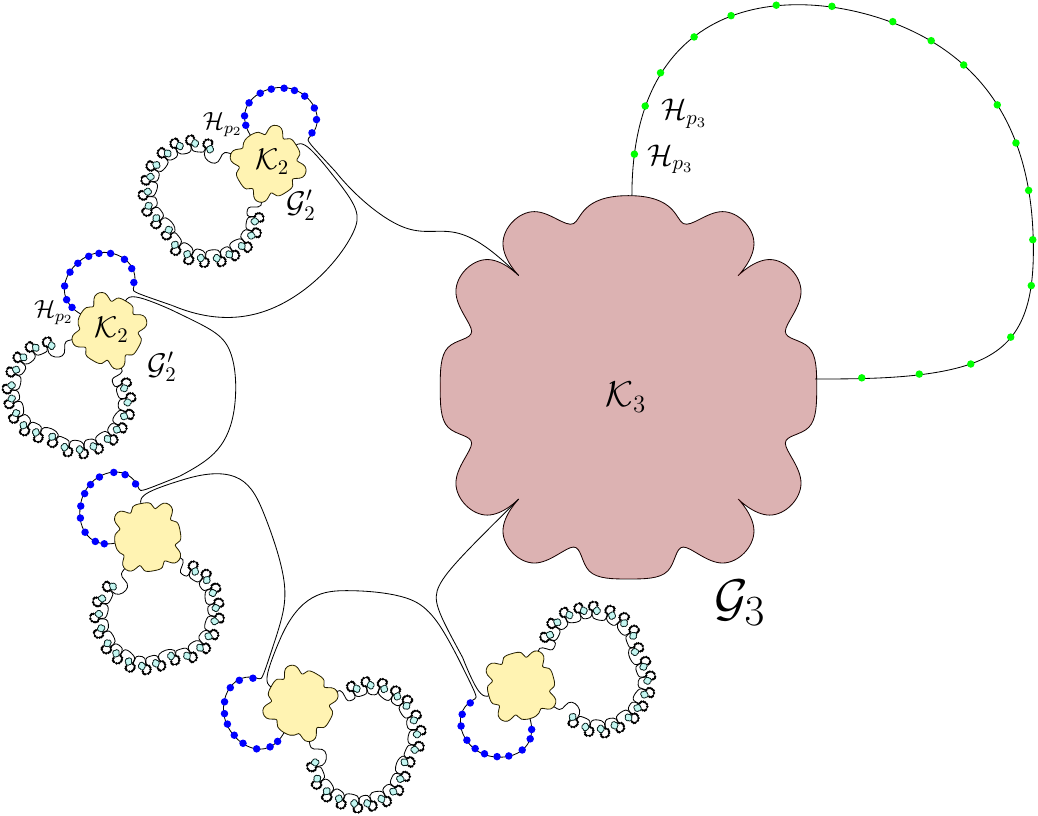}}
    \caption{A further iteration of the construction in \Cref{fig:intro}.}
    \label{fig:construction}
\end{figure}

We also introduce auxiliary notation, which will help in simplifying calculations in the proof of the \hyperref[mainthm]{Main Theorem}. 

\begin{notation}
    Let $\mathcal{G}$ be a graph of groups and let $\mathcal{G}'$ be a precover of $\mathcal{G}$. The \emph{pre-degree} of the precover $f:\mathcal{G}'\longrightarrow \mathcal{G}$ is defined as
    \[\coind{\mathcal{G}:\mathcal{G}'}= \max_{v\in \mathrm{V}(\Gamma)} \left(\sum_{v'\in f^{-1}(v)} [G_v:G_{v'}]\right).
    \]
    Intuitively, the pre-degree $\coind{\mathcal{G}:\mathcal{G}'}$ is a lower bound for the degree of any completion of $\mathcal{G}'$ into a cover of $\mathcal{G}$.
\end{notation}

\begin{rmk} \label{rmk:coind}
    If $\mathcal{G}$ is a hyperbolic graph of free groups with cyclic edges, and $\coind{\mathcal{G}:\mathcal{G}'}<\infty$, then $\mathcal{G}'$ can be completed to a finite-sheeted cover $\mathcal{H}\twoheadrightarrow \mathcal{G}$. In this case,
    \[
    \coind{\mathcal{G}:\mathcal{G}'}\le\deg(\mathcal{H}\twoheadrightarrow \mathcal{G})=[G:H].
    \]
    Moreover, if $\mathcal{G}_1,\ldots,\mathcal{G}_n$ are precovers of $\mathcal{G}$, and $\mathcal{G}'$ is obtained from $\mathcal{G}_1,\ldots,\mathcal{G}_n$ by pairing some of their hanging elevations, then
    \[
    \coind{\mathcal{G}:\mathcal{G}'}\le\coind{\mathcal{G}: \mathcal{G}_1} +\cdots+ \coind{\mathcal{G}:\mathcal{G}_n}.
    \]
    Similarly, if instead  $\mathcal{X}'$ is obtained by identifying vertex groups of different precovers in the collection $\mathcal{G}_1,\ldots,\mathcal{G}_n$, then
    \[\coind{\mathcal{G}:\mathcal{G}'}\le \coind{\mathcal{G}:\mathcal{G}_1}+\cdots+\coind{\mathcal{G}:\mathcal{G}_n}.\]
\end{rmk}

\begin{proof}[Proof of the {\hyperref[mainthm]{Main Theorem}}]

    For ease of bookkeeping, we continue with the following convention: calligraphic letters (e.g. $\mathcal{G}, \mathcal{H}, \mathcal{K}, \mathcal{L}$) denote graphs of groups, while the corresponding roman letters ($G,H,K,L$) denote their fundamental groups, with subscripts and superscripts matching accordingly. For the underlying graphs we use $\Gamma$ for $\mathcal{G}$ and $\Lambda$ for $\mathcal{L}$. \par \smallskip

    \begin{enumerate}[label=\textbf{Step \arabic*:}, leftmargin=0.6cm]

    \item \textbf{Preparations.} We begin by passing to a finite-index subgroup $G_0\le G$, which will serve as a convenient starting point for our construction. Since this adjustment affects the homological torsion growth of $G$ only by a multiplicative constant, it suffices to construct the sequence starting with $G_0$. \par \smallskip
    
    By \cite[Theorem 5.1]{wise:graph-sep} $G$ is  residually finite, and hence virtually torsion free. We may therefore replace $G$ with a finite-index subgroup that splits as a graph of free groups with cyclic edge groups, and which is not isomorphic to a free product of free and surface groups. As in the beginning of \Cref{sec:pieces}, $G$ admits a finite-index subgorup $G'$ that satisfies the four properties of \Cref{prop:normal_form}. \par \smallskip

    Fix a non-cyclic vertex $v$ of $\mathcal{G}'$. Finite-index subgroups of $G_v$ will serve as ``connectors'' in our construction, allowing us to connect torsion pieces and copies of previously constructed groups to precovers of $G$. Using \Cref{claim:connected}, we may assume that there are two cyclic vertices adjacent to $v$, whose removal does not disconnect $\Gamma'$. We call them $C_{\mathrm{tor}}$ and $C_{\mathrm{prev}}$. New torsion pieces will always be connected (via a connector) to lifts of $C_{\mathrm{tor}}$; copies of previously constructed groups will always be connected (again, via a connector) to lifts of $C_{\mathrm{prev}}$. \par \smallskip

    Let $(G'_v,[\underline{w}'_v=\{w_1,\ldots,w_m\}])$ be the induced pair at $G'_v$, so that the conjugacy class $[w'_1]$ is attached to $C_{\mathrm{tor}}$ via an edge $e_{\mathrm{tor}}\in \mathrm{E}(\Gamma')$, and $[w'_2]$ is attached to $C_{\mathrm{prev}}$ via $e_{\mathrm{prev}} \in \mathrm{E}(\Gamma')$. Since $G'$ is in normal form, the pair $(G'_v,[\underline{w}'_v])$ is malnormal, and we may use Wise's omnipotence as in \Cref{thm:omni} to generate finite-index pull-back subpairs of $(G'_v,[\underline{w}'_v])$ whose peripheral structure consists of elevations of $[\underline{w}'_v]$ of prescribed degrees. \par \smallskip

    Write $M$ for the omnipotence constant of the pair $(G'_v,[\underline{w}'_v])$, and let $(G_0)_v$ be the finite-index subgroup of $G'_v$ obtained from \Cref{thm:omni} for the collection of degrees $\underline{d}=(1,\ldots,1)$, so that every elevation of a conjugacy class in $[\underline{w}'_v]$ to $(G_0)_v$ is of degree $M$. The finite-index pull-back pair $((G_0)_v,[(\underline{w}_0)_v])$ satisfies a restricted version of omnipotence: for any pair of positive numbers $d_{\mathrm{tor}}$ and $d_{\mathrm{prev}}$, there exists a finite-index subgroup $(G_0)_v(d_{\mathrm{tor}}, d_{\mathrm{prev}})$ satisfying the following:
    \begin{enumerate}
        \item If $[w_{\mathrm{tor}}]$ is an elevation of $[w'_1]$ to $(G_0)_v(d_{\mathrm{tor}}, d_{\mathrm{prev}})$, then
        \[
        \deg_{(G_0)_v}([w_{\mathrm{tor}}])=d_{\mathrm{tor}}.
        \]
        \item Similarly, if $[w_\mathrm{prev}]$ is an elevation of $[w'_2]$ to $(G_0)_v(d_{\mathrm{tor}}, d_{\mathrm{prev}})$, then
        \[
        \deg_{(G_0)_v}([w_{\mathrm{prev}}])=d_{\mathrm{prev}}.
        \]
    \end{enumerate}
    By the refined version of Wise's omnipotence in \Cref{thm:omni}, the same will apply to any finite-index subgroup of $(G_0)_v$ obtained by further applications of omnipotence. \par \smallskip

    Finally, $G'$ is subgroup separable by \cite[Theorem 5.1]{wise:graph-sep}; we may therefore complete $(G_0)_v$ to a finite-index subgroup $G_0 \le G'$. We still refer to the vertex of $\mathcal{G}_0$ whose vertex group is $(G_0)_v$ by $v$. In preparation for the induction, let $p_1,p_2,\ldots$ be an enumeration of the primes, and let $g_1,g_2,\ldots$ be an enumeration of the elements of $G_0$. \par \smallskip
    
    \item \textbf{Base case: constructing $G_1$.} We remark that the construction of $G_1$ could in principle be simplified, but we present it in a form parallel to the inductive step for clarity. \par \smallskip

    \noindent \textbf{The base group $L_1$.} Let $L_1 \le G_0$ be a finite-index subgroup which excludes $g_1$, and such that there is a vertex $v^1\in \mathrm{V}(\Lambda_1)$ lying above $v$ with
    \begin{equation} \label{eq:dist}
        d_{\Lambda_1}(p_{\mathcal{L}_1}, v^1) > \norm{g_1}_{\mathcal{G}_0}+1.    
    \end{equation}
    Here, $\norm{g_1}_{\mathcal{G}_0}$ denotes the path length of $g_1$ in the underlying graph $\Gamma_0$, and $p_{\mathcal{L}_1}$ is the base vertex of $\mathcal{L}_1$. By \Cref{claim:lift-cutset}, there are two cyclic vertices $C^1_{\mathrm{tor}}$ and $C^1_{\mathrm{prev}}$ that lie above $C_{\mathrm{tor}}$ and $C_{\mathrm{prev}}$ respectively, which are adjacent to $v^1$ and do not form a cut set of $\Lambda_1$. Condition \ref{eq:dist} implies that if we disconnect $\mathcal{L}_1$ at $C^1_{\mathrm{tor}}$ and $C^1_{\mathrm{prev}}$, and attach in place other precovers of $\mathcal{G}_0$, $g_1$ will not lie in the resulting group.\par \smallskip
    
    Such a subgroup $L_1\le G_0$ always exists: indeed, by \Cref{claim:connected}, after possibly passing to a finite-index subgroup of $G_0$, we may assume that $\pi_1(\Gamma_0)$ is free of rank at least $2$. We may then choose a finite-sheeted normal cover $\Lambda_1 \twoheadrightarrow \Gamma_0$ in which a lift of the base vertex $p_{\mathcal{G}_0}$ and a lift of $v$ lie at distance greater than $\norm{g_1}_{\mathcal{G}}+1$. Replacing each vertex of $\Lambda_1$ by the corresponding vertex group of $\mathcal{G}_0$ produces a finite-sheeted cover $\mathcal{L}_1$ of $\mathcal{G}_0$. If $g_1 \in L_1$, residual finiteness allows us to pass to a further finite-index normal subgroup in which $g_1$ does not lie (and in which \ref{eq:dist} above is still satisfied). \par \smallskip

    \noindent \textbf{The setup.} Recall that in $\mathcal{G}'$, the conjugacy class $[w'_1]$ was attached to the cyclic vertex $C_{\mathrm{tor}}$ via the edge $e_{\mathrm{tor}}$, and the conjugacy class $[w'_2]$ was attached to the cyclic vertex $C_{\mathrm{prev}}$ via the edge $e_{\mathrm{prev}}$. We fix the following notation:
    \begin{enumerate}
        \item $e^1_{\mathrm{tor}}$ is the edge of $\mathrm{E}(\Lambda_1)$ that lies above $e_{\mathrm{tor}}$, and connects $v^1$ to $C^1_{\mathrm{tor}}$. The corresponding conjugacy class of the vertex group $(L_1)_{v^1}$ is denoted $[w^1_{\mathrm{tor}}]$, and its degree over the corresponding conjugacy class in $(G_0)_v$ is $d^1_{\mathrm{tor}}$.
        \item Similarly, the edge $e^1_{\mathrm{prev}}\in \mathrm{E}(\Lambda_1)$ lies above $e_{\mathrm{prev}}$ and connects $v^1$ to $C^1_{\mathrm{prev}}$. The corresponding conjugacy class of $(L_1)_{v^1}$ is denoted $[w^1_{\mathrm{prev}}]$, and its degree over $(G_0)_v$ is $d^1_{\mathrm{prev}}$.
    \end{enumerate}

    Note that since $\mathcal{G}'$ is in normal form, the stars of $C_{\mathrm{tor}}$ and $C_{\mathrm{prev}}$ are preserved in every finite-sheeted cover of $\mathcal{G}'$. We will therefore keep using the notation above in every step of the construction, replacing the superscript $1$ with $n$ at the $n$-th iteration. \par \smallskip

    \noindent \textbf{Constructing a $p_1$-torsion piece}. We apply \Cref{lem:beads} to $\mathcal{L}_1$. Let $\mathcal{H}_{p_1}$ be the $p_1$-torsion piece constructed from $\mathcal{L}_1$ with respect to the vertex group $C^1_{\mathrm{prev}}$. If $C^{p_1}_{1}$ and $C^{p_1}_2$ are its two cyclic vertices admitting hanging elevations, we choose $\mathcal{H}_{p_1}$ so that the only hanging elevation at $C^{p_1}_1$ corresponds to the edge $e^1_{\mathrm{prev}}$. Denote its degree, over $\mathcal{L}_1$, by $d_{p_1}$. \par \smallskip

    \noindent \textbf{Assembling $\mathcal{G}_1$.} We continue by disconnecting the edge $e^1_{\mathrm{tor}}$ from $v^1$, and introducing a new non-cyclic vertex $v^1_{\mathrm{tor}}$ that will be connected to $C^1_{\mathrm{tor}}$ in its place. By \Cref{thm:omni}, there is a finite-index subgroup $G^1_{\mathrm{tor}}$ of $(G_0)_v$ such that
    \begin{enumerate}
        \item every elevation of $[w^1_{\mathrm{tor}}]$ to $G^1_{\mathrm{tor}}$ is of degree $d^1_{\mathrm{tor}}$, and
        \item every elevation of $[w^1_{\mathrm{prev}}]$ to $G^1_{\mathrm{tor}}$ is of degree $d_{p_1}$.
    \end{enumerate}
    This implies that we can connect $G^1_{\mathrm{tor}}$ to $C^1_{\mathrm{tor}}$ via an elevation of $[w^1_{\mathrm{tor}}]$, and to $\mathcal{H}_{p_1}$ via an elevation of $[w^1_{\mathrm{prev}}]$, as in \Cref{fig:G1}. The resulting precover of $\mathcal{G}_0$ is denoted $\mathcal{K}'_1$. Note that $\mathcal{K}'_1$ is connected by \Cref{claim:lift-cutset}. \par \smallskip

    By \cite[Theorem 5.1]{wise:graph-sep}, we may complete $\mathcal{K}'_1$ to a finite-sheeted cover $\mathcal{K}''_1$ of $\mathcal{G}_0$. Let $\mathcal{K}_1$ be the (possibly disconnected) precover of $\mathcal{G}_0$ obtained by removing $\mathcal{H}_{p_1}$ from $\mathcal{K}''_1$. The last step in the construction of $G_1$ is to attach to $\mathcal{K}_1$ a long ``chain'' of copies of $\mathcal{H}_{p_1}$, obtained by identifying the cyclic vertex $C^{p_1}_1$ of one copy with the cyclic vertex $C^{p_1}_2$ of the next. Take $\alpha$ copies of $\mathcal{H}_{p_1}$ so that
    \[
    \frac{\alpha \cdot \coind{\mathcal{G}_0:\mathcal{H}_{p_1}}}{\alpha \cdot \coind{\mathcal{G}_0:\mathcal{H}_{p_1}} + \coind{\mathcal{G}_0:\mathcal{K}_1} } \ge \frac{1}{2^{1+1}} = \frac{1}{4}. 
    \]
    To assemble $G_1$ simply splice these $\alpha$ copies of $\mathcal{H}_{p_1}$ and join them to $\mathcal{K}_1$ as in \Cref{fig:G1}. Note that by our choice of $\mathcal{L}_1$, we have that $g_1 \notin G_1$. \par \smallskip

    \begin{figure}[h]
        \centering
        \includegraphics[width=0.75\linewidth]{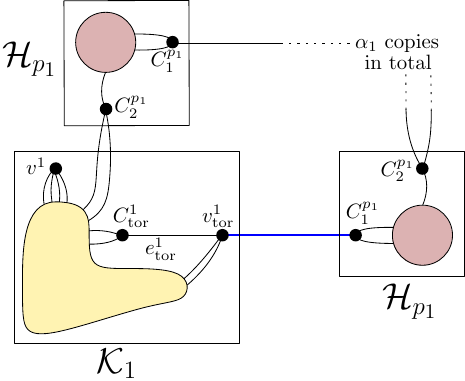}
        \caption{The construction of $\mathcal{G}_1$. The irrelevant parts of $\mathcal{K}_1$ and of the copies of $\mathcal{H}_{p_1}$ appear in yellow and red. Note that the blue edge, connecting $v^1_{\mathrm{tor}}$ to $\mathcal{H}_{p_1}$, is a lift of $e_{\mathrm{prev}}$.}
        \label{fig:G1}
    \end{figure}
    
    \noindent \textbf{Computing homological torsion.} By the second item of \Cref{lem:beads}, each copy of $\mathcal{H}_{p_1}$ contributes direct factor which is a $p_1$-group to homology. A standard Mayer-Vietoris argument then gives
    \begin{align*}
    \frac{\log_{p_1} \abs{ \mathrm{Tor}_p(G_1^{\mathrm{ab}})}}{[G_0:G_1]} & \ge\frac{\log_{p_1} \abs{ \mathrm{Tor}_p((\mathbb{Z}/p_1\mathbb{Z})^\alpha)}}{[G_0:G_1]} 
    \\ & = \frac{\alpha}{\alpha \cdot \coind{\mathcal{G}_0:\mathcal{H}_{p_1}} + \coind{\mathcal{G}_0:\mathcal{K}_1}} 
    \\ &\ge \left(\frac{1}{2^{1+1}}\right)\cdot \frac{1}{\coind{\mathcal{G}_0:\mathcal{H}_{p_1}}}.
    \end{align*}

    \item \textbf{The induction hypothesis.} We clearly lay out the inductive assumptions. Suppose we have constructed the groups $G_0,\ldots,G_{n-1}$, such that for every $k\le n-1$ the following holds:
    \begin{enumerate}
        \item $g_1,\ldots, g_k \notin G_k$.
        \item $\mathcal{G}_k$ is assembled from the following precovers of $\mathcal{G}_{k-1}$:
        \begin{enumerate}
            \item $\mathcal{K}_k$, a (possibly disconnected) precover that contains two non-cyclic vertices $v^k_{\mathrm{tor}}$ and $v^k_{\mathrm{prev}}$ that lie above $v$. The vertex groups corresponding to $v^k_{\mathrm{tor}}$ and $v^k_{\mathrm{prev}}$ were obtained using omnipotence, so that in each of them all of the elevations of $[w_{\mathrm{tor}}]$ share the same degree; similarly, in each of them, all of the elevations of $[w_{\mathrm{prev}}]$ share the same degree.
            \item A chain of $\alpha_k$ copies of a $p_k$-torsion piece $\mathcal{H}_{p_k}$, a precover of $\mathcal{G}_{k-1}$, constructed using \Cref{lem:beads}.
            \item For $k>1$, a chain of $\beta_k$ copies of $\mathcal{G}_{k-1}$, each obtained by detaching an edge $e^{k-1}_{\mathrm{tor}}$ from the cyclic vertex $C^{k-1}_{\mathrm{tor}}$ that lies above $C_{\mathrm{tor}}$ in $\mathcal{G}_{k-1}$.
        \end{enumerate}
        The first chain is attached to $\mathcal{K}_k$ via $v^k_{\mathrm{tor}}$, and the second chain is attached to $\mathcal{K}_k$ via $v^k_{\mathrm{prev}}$. We refer the reader to \Cref{fig:construction} which depicts this configuration.
        \item For every $i\le k$, 
         \begin{equation} \label{ineq:IH}
        \frac{\log_{p_i}\abs{\mathrm{Tor}_{p_i}(G_k)}}{[G:G_k]}\ge \frac{\prod_{j=i+1}^k\bigg(1-\frac{1}{2^j}\bigg)}{2^{i+1}\cdot \coind{\mathcal{G}_0:\mathcal{H}_{p_i}}}
        \end{equation}
        
    \end{enumerate}

    \item \textbf{Inductive step.} 
    The construction is similar to that in the base case, with one change: we need to preserve the torsion constructed in the previous stages, so $\mathcal{G}_n$ will also contain copies of the previously constructed $\mathcal{G}_{n-1}$. \par \smallskip
    
    \noindent \textbf{The base group $L_n$}. We begin by taking a finite-index subgroup $L_n\le G_{n-1}$ which excludes $g_1,\ldots,g_n$. As in the base case, we want to guarantee that if we disconnect $\mathcal{L}_n$ at the cyclic vertices $C^n_{\mathrm{tor}}$ and $C^n_{\mathrm{prev}}$, and connect other precovers of $\mathcal{G}_{n-1}$ in place,  the elements $g_1,\ldots,g_n$ will still be excluded from the resulting group. To do so, we pick a finite-index subgroup $L_n \le G_{n-1}$ such that
 
    \begin{enumerate}
        \item $g_1,\ldots,g_n \notin L_n$.
        \item Let $p_{\mathcal{L}_n}\in \mathrm{V}(\Lambda_n)$ be the base vertex of $\mathcal{L}_n$. Then there is a vertex $v^n\in \mathrm{V}(\Lambda_n)$ that lies above $v^{n-1}_{\mathrm{tor}}$ such that
        \[
        d_{\Lambda_n}(p_{\mathcal{L}_n},v^n)> \max\{\norm{g_1}_{\mathcal{G}_0},\ldots,\norm{g_n}_{\mathcal{G}_0}\} + 1.
        \]
    \end{enumerate}
    The argument used to obtain $\mathcal{L}_1$ from $\mathcal{G}_0$ in the base case shows that such a subgroup exists.
    \par \smallskip

    \noindent \textbf{The setup.} As in the base case, we fix the following notation:
    \begin{enumerate}
        \item $C^n_{\mathrm{tor}}$ and $C^n_{\mathrm{prev}}$ are cyclic vertices of $\mathcal{L}_n$ that are adjacent to $v^n$ and which lie above $C^{n-1}_{\mathrm{tor}}$ and $C^{n-1}_{\mathrm{prev}}$. By \Cref{claim:lift-cutset}, we may assume that removing them from $\mathcal{L}_n$ does not disconnect $\Lambda_n$.
        \item $e^n_{\mathrm{tor}}\in \mathrm{E}(\Lambda_n)$ is the edge that connects $v^n$ to $C^n_{\mathrm{tor}}$ and which lies above $e_{\mathrm{tor}}$. We denote the corresponding conjugacy class in $(L_n)_{v^n}$ by $[w^n_{\mathrm{tor}}]$, and its degree over the corresponding conjugacy class in the vertex group $(G_{n-1})_{v^{n-1}_{\mathrm{tor}}}$ of $\mathcal{G}_{n-1}$ by $d^n_{\mathrm{tor}}$.
        \item Similarly, the edge $e^n_{\mathrm{prev}}$ connects $v^n$ to $C^n_{\mathrm{prev}}$ and lies above $e_{\mathrm{prev}}$. Its corresponding conjugacy class in $(L_n)_{v^n}$ is denoted $[w^n_{\mathrm{prev}}]$, and its degree over $(G_{n-1})_{v^{n-1}_{\mathrm{tor}}}$ is denoted $d^n_{\mathrm{prev}}$.
    \end{enumerate}

    \noindent \textbf{Constructing a $p_n$-torsion piece.} Using \Cref{lem:beads}, we construct a $p_n$-torsion piece $\mathcal{H}_{p_n}$ out of $\mathcal{L}_n$, with respect to the cyclic vertex group $C^n_{\mathrm{prev}}$. Denote by $C^{p_n}_1$ and $C^{p_n}_2$ its cyclic vertex groups that admit hanging elevations. By construction, we may assume that the only hanging elevation at $C^{p_n}_1$ comes from the edge $e^n_{\mathrm{prev}}$.  We denote its degree, over the corresponding edge map in $\mathcal{L}_n$, by $d_{p_n}$. \par \smallskip

    \noindent \textbf{The connectors $G^n_{\mathrm{tor}}$ and $G^n_{\mathrm{prev}}$.} We next construct two finite-index subgroups $G^n_{\mathrm{tor}}$ and $G^n_{\mathrm{prev}}$ of the non-cyclic vertex group $G^{n-1}_{\mathrm{tor}}=(G_{n-1})_{v^{n-1}_{\mathrm{tor}}}$ of $\mathcal{G}_{n-1}$; this group played the role of $G^n_{\mathrm{tor}}$ in the previous step. These groups will be used to attach copies of $\mathcal{G}_{n-1}$ and $\mathcal{H}_{p_n}$ to the cyclic vertices $C^n_{\mathrm{prev}}$ and $C^n_{\mathrm{tor}}$. As in the base case, $v^n_{\mathrm{tor}}$ will be attached to the $p_n$-torsion piece $\mathcal{H}_{p_n}$ and to the cyclic vertex $C^n_{\mathrm{tor}}$. On the other hand, the base case does not include a construction analogous to that of $v^n_{\mathrm{prev}}$. This vertex will be connected to a copy of $\mathcal{G}_{n-1}$, and to the cyclic vertex $C^n_{\mathrm{prev}}$. \par \smallskip

    We begin with $G^n_{\mathrm{tor}}$, and let it be a finite-index subgroup of $G^{n-1}_{\mathrm{tor}}$ such that:
    \begin{enumerate}
        \item every elevation of $[w^1_{\mathrm{prev}}]$ to $G^n_{\mathrm{tor}}$ is of degree $d_{p_n}\cdot d^n_{\mathrm{prev}}$ over the corresponding conjugacy class in $G^{n-1}_{\mathrm{tor}}$, and
        \item every elevation of $[w^1_{\mathrm{tor}}]$ to $G^n_{\mathrm{tor}}$ is of degree $d^n_{\mathrm{tor}}$ (again, over the corresponding conjugacy class in $G^{n-1}_{\mathrm{tor}}$).
    \end{enumerate}
    The existence of such a subgroup is guaranteed by the refined version of omnipotence in \Cref{thm:omni}. Note that the chosen degrees imply that we may connect $G^n_{\mathrm{tor}}$ to the cyclic vertex $C^{p_n}_1$ of $\mathcal{H}_{p_n}$ via an elevation of $[w^1_{\mathrm{prev}}]$. Similarly, if we disconnect the edge $e^n_{\mathrm{tor}}$ from $v^n$ in $\mathcal{L}_n$, we may connect $G^n_{\mathrm{tor}}$ in place. \par \smallskip

    The construction of $G^n_{\mathrm{prev}}$ is similar: using \Cref{thm:omni}, it can be taken a finite-index subgroup of $G^{n-1}_{\mathrm{tor}}$ such that
    \begin{enumerate}
        \item every elevation of $[w^1_{\mathrm{prev}}]$ to $G^n_{\mathrm{tor}}$ is of degree $d^n_{\mathrm{prev}}$ over the corresponding conjugacy class in $G^{n-1}_{\mathrm{tor}}$, and
        \item every elevation of $[w^1_{\mathrm{prev}}]$ to $G^n_{\mathrm{tor}}$ is of degree $1$ (again, over the corresponding conjugacy class in $G^{n-1}_{\mathrm{tor}}$).
    \end{enumerate}

    As before, if we disconnect the edge $e^n_{\mathrm{prev}}$ from $v^n$ in $\mathcal{L}_n$, we may connect $G^n_{\mathrm{prev}}$ in place. Moreover, choose an edge of $\mathcal{G}_{n-1}$ that lies below $e^n_{\mathrm{tor}}$ and that is adjacent to $v^{n-1}_{\mathrm{tor}}$. Disconnect it from its incident cyclic vertex, and denote the resulting precover of $\mathcal{G}_{n-1}$ by $\mathcal{G}'_{n-1}$. We may now connect $G^n_{\mathrm{prev}}$ to $\mathcal{G}'_{n-1}$. We point out that the precover $\mathcal{G}'_{n-1}$ will be used later in the construction, and that $\coind{\mathcal{G}_{n-1}:\mathcal{G}'_{n-1}}=1$. \par \smallskip

    \noindent \textbf{Assembling $\mathcal{G}_n$.} Disconnect the edges $e^n_{\mathrm{tor}}$ and $e^n_{\mathrm{prev}}$ from $v^n$, and connect to $C^n_{\mathrm{tor}}$ and $C^n_{\mathrm{prev}}$ instead:
    \begin{enumerate}
        \item A new non-cyclic vertex $v^n_{\mathrm{tor}}$ with vertex group $G^n_{\mathrm{tor}}$. $v^n_{\mathrm{tor}}$ is connected to $C^n_{\mathrm{tor}}$ via $e^n_{\mathrm{tor}}$.
        \item A new non-cyclic vertex $v^n_{\mathrm{prev}}$ with vertex group $G^n_{\mathrm{prev}}$. Likewise, $v^n_{\mathrm{prev}}$ is connected to $C^n_{\mathrm{prev}}$ (via $e^n_{\mathrm{prev}}$).
    \end{enumerate}
    Recall that \Cref{claim:lift-cutset} implied that $C^n_{\mathrm{tor}}$ and $C^n_{\mathrm{prev}}$ did not form a cut set of $\Lambda_n$. This implies that the resulting precover of $\mathcal{L}_n$ is connected. Continue by attaching a copy of $\mathcal{H}_{p_n}$ to $v^n_{\mathrm{tor}}$, and a copy of $\mathcal{G}_{n-1}$ to $v^n_{\mathrm{prev}}$, as laid out in the previous step. This results in a connected precover $\mathcal{K}'_n$ of $\mathcal{G}_{n-1}$. Subgroup separability now implies that we may complete $\mathcal{K}'_n$ to a finite-sheeted cover $\mathcal{K}''_n$ of $\mathcal{G}_{n-1}$. The (possibly disconnected) precover $\mathcal{K}_n$ of $\mathcal{G}_{n-1}$ is obtained by removing $\mathcal{H}_{p_n}$ and $\mathcal{G}'_{n-1}$.  \par \smallskip

    \begin{figure}[h!]
        \centering
        \makebox[\textwidth][c]{\includegraphics[width=1.2\textwidth]{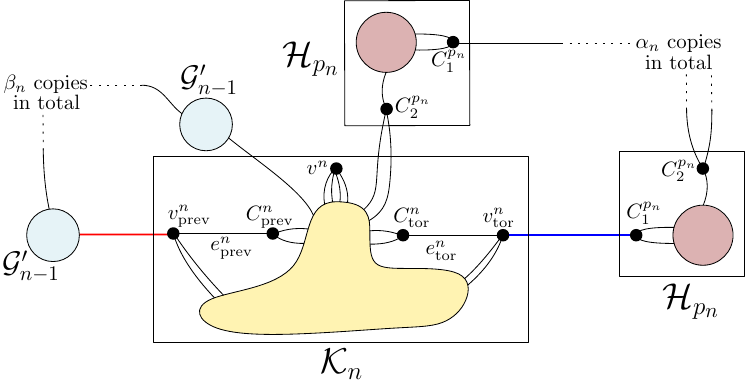}}
        \caption{The construction of $\mathcal{G}_n$. The irrelevant parts of $\mathcal{K}_n$ and of the copies of $\mathcal{H}_{p_n}$ appear in yellow and red; the copies of $\mathcal{G}'_{n-1}$ are coloured blue. Note that the blue edge connecting $v^n_{\mathrm{tor}}$ to a copy of $\mathcal{H}_{p_n}$ lies above $e_{\mathrm{prev}}$, and the red edge connecting $v^n_{\mathrm{prev}}$ to a copy of $\mathcal{G}'_{n-1}$ lies above $e_{\mathrm{tor}}$.}
        \label{fig:Gn-pieces}
    \end{figure}
    
    To construct $G_n$, we attach two long ``chains'' to $\mathcal{K}_n$:
    \begin{enumerate}
        \item A chain of copies of $\mathcal{G}'_n$, that will ensure that the previously constructed homological torsion is preserved in $G_n$.
        \item A chain of copies of $\mathcal{H}_{p_n}$, that will introduce $p_n$-homological torsion to $G_n$.
    \end{enumerate}
    \smallskip
    Gather $\alpha_n$ copies of $\mathcal{H}_{p_n}$ and $\beta_n$ copies of $\mathcal{G}'_{n-1}$ so that
    \[
    \frac{\alpha_n \cdot \coind{\mathcal{G}_{n-1}:\mathcal{H}_{p_n}}}{\alpha_n \cdot \coind{\mathcal{G}_{n-1}:\mathcal{H}_{p_n}} + \beta_n \cdot \coind{\mathcal{G}_{n-1}:\mathcal{G}'_{n-1}} + \coind{\mathcal{G}_{n-1}:\mathcal{K}_n} } \ge \frac{1}{2^{n+1}} 
    \]
    and
    \begin{equation*}
         \frac{\beta_n}{[G_{n-1}:G_n]} \ge \frac{\beta_n \cdot \coind{\mathcal{G}_{n-1}:\mathcal{G}'_{n-1}}}{\alpha_n \cdot \coind{\mathcal{G}_{n-1}:\mathcal{H}_{p_n}} + \beta_n \cdot \coind{\mathcal{G}_{n-1}:\mathcal{G}'_{n-1}} + \coind{\mathcal{G}_{n-1}:\mathcal{K}_n}}  \ge 1-\frac{1}{2^n}.
    \end{equation*}

    Splice the $\alpha_n$ copies of $\mathcal{H}_{p_n}$ and $\beta_n$ copies of $\mathcal{G}'_{n-1}$ to form two chains, and attach these chains to $\mathcal{K}_n$. This results in a connected cover $\mathcal{G}_n$ of $\mathcal{G}_{n-1}$. We hope that \Cref{fig:Gn-pieces} will elucidate this construction. \par \smallskip
    
    \noindent \textbf{Homological torsion.} We compute the homological torsion of $\mathcal{G}_n$ at the primes $p_1,\ldots,p_n$. Note that by construction, the inequality $\ref{ineq:IH}$ from the induction hypothesis also applies to $\mathcal{G}'_{n-1}$. Hence, for every $k<n$,
    \begin{align*}
    \frac{\log_{p_k}\abs{\mathrm{Tor}_{p_k}(G_n^{\ab})}}{[G_0:G_n]} &\ge \frac{\beta_n \cdot \log_{p_k}\abs{\mathrm{Tor}_{p_k}(G_{n-1}^{\ab})}}{[G_0:G_{n-1}]\cdot [G_{n-1}:G_n]} 
    \\ & = \left(\frac{\beta_n}{[G_{n-1}:G_n]}\right)\cdot \left(\frac{\log_{p_k}\abs{\mathrm{Tor}_{p_k}(G_{n-1}^{\ab})}}{[G_0:G_{n-1}]}\right) 
    \\ &\ge \left(1-\frac{1}{2^n}\right)  \cdot \frac{\prod_{j=k+1}^n\left( 1- \frac{1}{2^j}\right)}{2^{k+1}\cdot \coind{\mathcal{G}_0:\mathcal{H}_{p_k}}} 
    \\ &=  \frac{\prod_{j=k+1}^{n+1}\left( 1- \frac{1}{2^j}\right)}{2^{k+1}\cdot \coind{\mathcal{G}_0:\mathcal{H}_{p_k}}} 
    \end{align*}

    Recall that by \Cref{lem:beads}, $H_{p_n}^{\ab}$ contains a direct factor which is a $p_n$-group, and which persists after gluing $\mathcal{H}_{p_n}$ to other precovers of $\mathcal{G}_{n-1}$. Therefore, at $p_n$,
    \begin{align*}
    \frac{\log_{p_n}\abs{\mathrm{Tor}_{p_n}(G_n^{\ab})}}{[G_0:G_n]} &\ge \frac{\alpha_n \cdot \log_{p_n}\abs{\mathrm{Tor}_{p_n}(H_{p_n}^{\ab})}}{[G_0:G_n]} 
    \\ & \ge \frac{\alpha_n}{[G_0:G_{n-1}]\cdot[G_{n-1}:G_n]}
    \\ & \ge \frac{\alpha_n}{[G_0:G_{n-1}]\cdot(\alpha_n \cdot \coind{\mathcal{G}_{n-1}:\mathcal{H}_{p_n}} + \beta_n \cdot \coind{\mathcal{G}_{n-1}:\mathcal{G}_{n-1}-C_{\mathcal{G}_{n-1}}} + \coind{\mathcal{G}_{n-1}:\mathcal{K}_n})}
    \\ &\ge \frac{1}{2^{n+1}\cdot \coind{\mathcal{G}_0:\mathcal{H}_{p_n}}},
    \end{align*}
    as required. 

    \item Finally, taking the limit $n\rightarrow \infty$, for every $k$,
    
    \[
    \liminf _{n \rightarrow \infty} \frac{\log_{p_k} \abs{\mathrm{Tor}_{p_k}(G_n^{\mathrm{ab}})}}{[G:G_n]}\ge \lim_{n\rightarrow \infty}  \frac{\prod_{j=k+1}^n\left (1-\frac{1}{2^j}\right)}{2^{k+1}\cdot \coind{\mathcal{G}_0:\mathcal{H}_{p_k}}}>0.  
    \]
    \end{enumerate}
\end{proof}

\bibliographystyle{plain}

\vspace{1cm}

(D. Ascari) \textsc{Department of Mathematics, University of the Basque Country, Barrio Sarriena, Leioa, 48940, Spain}

\emph{Email address:} \texttt{ascari.maths@gmail.com} \par \smallskip
\vspace{0.5cm}
(J. Fruchter) \textsc{Mathematisches Institut, Rheinische Friedrich-Wilhelms-Universit\"at Bonn, Endenicher Allee 60, 53115 Bonn, Germany}

\emph{Email address:} \texttt{fruchter@math.uni-bonn.de}
\end{document}